\documentclass[nofootinbib,aip,jmp,a4paper,11pt,reprint,onecolumn]{revtex4-1}

\usepackage[utf8]{inputenc} 
\usepackage[T1]{fontenc} 
\usepackage[english]{babel}
\usepackage{xcolor,color}
\usepackage[hidelinks]{hyperref}

\usepackage{amsmath,latexsym,mathrsfs,bbold,
	amsbsy,amsfonts,amssymb,amsthm,amscd,amsxtra,amstext,amsopn,dsfont,yfonts,geometry}
\usepackage{mathtools,emptypage}

\newtheorem{theorem}{Theorem}
\newtheorem{corollary}[theorem]{Corollary}
\newtheorem{lemma}[theorem]{Lemma}

\newcommand{\toitself}{\mathbin{\scalebox{.85}{%
			\lefteqn{\scalebox{.5}{$\blacktriangleleft$}}\raisebox{.34ex}{$\supset$}}}}

\newcommand{\ZZZ}{\mathds{Z}}

\newcommand{\RRR}{\mathds{R}} 
\newcommand{\TTT}{\mathds{T}}
 
\newcommand{\ul}{\underline}

\newcommand{\e}{\varepsilon} 
\newcommand{\om}{\omega} 

\begin{document}

\title{\bf Analiticity of the Lyapunov exponents of perturbed toral
automorphisms.}
	
\author{Gian Marco Marin}
\affiliation{Dipartimento di Matematica e Informatica, Universit\`a
degli studi di Trieste, Trieste, Italy}
\author{Federico Bonetto}
\affiliation{School of Mathematics, Georgia Institute of Technology,
Atlanta GA.}
\author{Livia Corsi}
\affiliation{Dipartimento di Matematica e Fisica, Universit\`a Roma Tre,
Roma, Italy.}

\date{\today}

\begin{abstract}
We consider a dynamical system generated by an analytic perturbation
$A_\varepsilon$ of an
analytic Anosov diffeomorphism $A_0$ of $\TTT^d$. We show that, if $A_0$ admit a
splitting of $\mathrm T\mathds T^d$ in $k$ invariant subspaces, there
exists a {\it partial conjugation} $\mathcal H_\e$ of $dA_\e$ and $dA_0$ that
preserves the splitting and is analytic in $\e$. This show that the splitting
can be extended to $A_\e$. As an application of this results, we obtain that the
Lyapunov exponents, if non degenerate, are analytic functions of the
perturbation.
\end{abstract}

\maketitle

\section{Introduction}\label{comparison}

Hyperbolicity is one the main ingredient in the modern study of long time
behavior of dynamical systems. Lyapunov exponents characterize in a quantitative
way the hyperbolic nature of a dynamical system. The knowledge of the Lyapunov
exponents, gives further insight on the dynamical properties of the system.
However, notwithstanding their importance, in practice it is very hard to obtain
rigorous quantitative information on Lyapunov exponents (see
Refs.~\onlinecite{Lieb,Ru} for
interesting examples). Most known results originate from numerical simulations
based on the classical algorithm introduced in Ref.~\onlinecite{BGGS}. Even the
question of
their regularity in the parameters defining a particular family of systems is
largely outside analytical reach (see e.g. fig. 12 in Ref.~\onlinecite{EkRu}).
It is thus
already interesting to obtain results on very simple systems, e.g. perturbations
of analytic diffeomorphisms of the torus in $d$ dimension. In this respect the
present work is an application of the techniques developed in
Refs.~\onlinecite{Falco,GBG} (see in particular Chapter 10) to the problem
studied in
Ref.~\onlinecite{Ru-per}. In the remainder of this introduction we provide two
examples
where we hope our results will turn out to be a starting point for further
research.

\medskip

Another important quantity that plays a role in the analysis of chaotic
dynamical systems is the fractal dimension of their attractor. There are several
different definitions of fractal dimension (see e.g. Ref.~\onlinecite{Young})
but a widely
accepted conjecture, called the Kaplan-Yorke (KY) conjecture, asserts that the
Hausdorff dimension $\mathrm{dim_{HD}}(\mu_A)$ of the SRB measure $\mu_A$ of an
hyperbolic dynamical system $A$ on a $d$-dimensional manifold is equal to its
Lyapunov dimension $\mathrm{dim_{L}}(\mu_A)$ introduced in
Refs.~\onlinecite{KY,FKY}. In the
case of a small perturbation of a system that preserves the Lebesgue measure the
Lyapunov dimension is
\begin{equation}\label{KY}
\mathrm{dim_{L}}(\mu_A)=d-\frac{\sigma}{\lambda_{1}}
\end{equation}
where $\lambda_1$ is the smallest Lyapunov exponent while $\sigma$ is the
average phase space contraction rate and it is equal to the sum of all Lyapunov
exponents. The KY conjectures was proved in Ref.~\onlinecite{Young} for systems
in 2
dimension and in Ref.~\onlinecite{LeYou} for random dynamical systems. On the
other hand,
it is easy to see that the KY conjecture cannot be true in general. Indeed we
can consider the direct product of two different independent Anosov systems on
$\mathds T^2$ for each of which the result in Ref.~\onlinecite{Young} apply; in
this case
the Hausdorff dimension of the SRB measure of the system, seen as a
diffeomorphism of $\TTT^4$, is the sum of the dimensions of the two components,
and a direct computation shows that this is in general different from the r.h.s.
of \eqref{KY}, see Section \ref{concl} for more details. Clearly a direct
product is not a generic system. However it would be enough to show that the
Hausdorff dimension of the SRB measure is continuous under perturbations to show
that the KY conjecture in generically false in an open neighbourhood of the
direct product of two independent system. A somehow similar argument is
developed in Ref.~\onlinecite{SBT}. For more details on this example, see the
discussion in
Section \ref{concl}, where we formulate a slightly different conjecture.
\medskip

%

\section{Definitions and results}\label{defi}

We consider an analytic diffeomorphism $A_0:\mathds T^d \toitself$ and
assume that it defines an Anosov dynamical system. This means that:

\begin{itemize}

\item there exists $\psi\in\mathds{T}^d$ such that the set $\{ A_{0}^m(\psi)
\colon \ m\in\mathds{Z}\}$ is dense;

\item the tangent space $\mathrm T_\psi\mathds T^d$ at $\psi\in\mathds{T}^d$ can
be written as the direct sum of two subspaces $W^{u}_{0}(\psi), \
W^{s}_{0}(\psi)$, with $\dim W^{s,u}_{0}(\psi)=d^{s,u}$, invariant under the
action of the differential $DA_{0}$ of $A_{0}$, that is
\begin{equation*}
DA_{0}(\psi)W^{s,u}_{0}(\psi)=W^{s,u}_{0}
(A_{0}(\psi));
\end{equation*}

\item $W^{s,u}_{0}(\psi)$ depends continuously on $\psi$;

\item the splitting $W^{s,u}_{0}(\psi)$ is \emph{hyperbolic}, in the sense that
there exist constants $  \Theta>0$ and $  0<\lambda<1$ such that for every $
n\geq 0$
\begin{equation}\label{hyper}
\begin{aligned}
\|DA_{0}^n(\psi)\mathbf{w}\|&\leq \Theta\lambda^n\|\mathbf{w}\|,
\qquad \mathbf{w}\in W^{s}_{0}(\psi)\\
\|DA_{0}^{-n}(\psi)\mathbf{w}\|&\leq
\Theta\lambda^n\|\mathbf{w}\|,
\qquad  \mathbf{w}\in W^{u}_{0}(\psi).
\end{aligned}
\end{equation}

\end{itemize}

To construct a perturbed system, we consider the lifting $\widetilde A_0$ of
$A_0$ to a map from $\mathds T^d$ to its universal covering $\mathds R^d$ and an
analytic function $F\in C^{\omega}(\mathds{T}^d,\mathds{R}^d)$. We then define
the one parameter family of diffeomorphisms
\begin{equation}\label{A}
A_{\varepsilon}(\psi)=\widetilde A_0(\psi)+\varepsilon F(\psi) \quad \mod
2\pi,
\end{equation}
where $\e$ is small and $\mod 2\pi$ represents the standard projection from
$\mathds R^d$ to $\mathds T^d$.

Given $f:\mathds{T}^d\mapsto\mathds{R}^k$ and $0<\beta<1$, we define the
\emph{H\"{o}lder  seminorm} of $f$ as
\begin{equation*}
|f|_{\beta}=\sup_{\psi,\psi'\in\mathds{T}^d}\frac{\|f(\psi)-f(\psi')\|}{\delta(
\psi,\psi'
)^{\beta}}.
\end{equation*}
where $\delta(\psi,\psi')$ is the metric inherited by $\mathds{T}^d$ from the
euclidean metric on $\mathds{R}^d$. Moreover we set
\begin{equation*}
\| f \|_{\beta}=\| f \|_{\infty} + | f|_{\beta}
\end{equation*}
where, as usual,  $\|f \|_{\infty}=\sup\substack{\psi\in\mathds{T}^d} \|
f(\psi) \|$. If $\|f\|_{\beta}$ is finite we say that $f$ is
($\beta$-)\emph{H\"{o}lder continuous}.

From structural stability it follows that, for $\varepsilon$ small enough, the
dynamical system $(\mathds{T}^d,A_{\varepsilon})$ is still Anosov. More
precisely, it follows from Ref.~\onlinecite{Falco} that, still for $\varepsilon$
small
enough, there exists a \emph{conjugation} $H_\varepsilon$ between the
\emph{perturbed} system $(\mathds{T}^d,A_{\varepsilon})$ and the
\emph{unperturbed} one $(\mathds{T}^d,A_0)$, that is a solution of the equation
\begin{equation}\label{H}
H_\varepsilon\circ A_0=A_{\varepsilon}\circ H_\varepsilon.
\end{equation}
Moreover $H_{\varepsilon}$ is an homeomorphism, H\"{o}lder
continuous in $\psi$, analytic in $\varepsilon$ with $H_0=\mathrm{Id}$.

Let $dA_\e$ be the differential of $A_\e$ seen as maps from $\mathrm T\mathds
T^d$ into itself, that is $dA_\e(\psi,v)=(A_\e(\psi),DA_\e(\psi)v)$ for $v\in
\mathrm T_\psi\TTT$. One may ask whether the conjugation $H_\e$ can be extended
to a conjugation $\mathcal{H}_\e(\psi,v)=(H_\e(\psi),\mathcal V_\e(\psi)v)$
between $dA_0$ and $dA_\e$. Since in general, the rates of expansion and
contraction of $A_0$ and $A_\e$ are different, it is not hard to see that such a
conjugation does not exist.

A more natural way to proceed is to assume that, for every $\psi$, $\mathrm
T_\psi\TTT^d$ can be written as the direct sum of $k$ subspaces
$V_1(\psi),\ldots,
V_k(\psi)$ invariant under $dA_0$, that is
\begin{equation}\label{vettd}
DA_0(\psi)V_i(\psi)=V_i(A_0(\psi))\quad  \mathrm{for}\quad i=1,\ldots, k.
\end{equation}
Moreover we assume that the $V_i$ are
$\beta$--H\"older continiuous for some
$\beta>0$, this means that $V_i(\psi)$ admits an orthonormal basis $\mathbf
v_i(\psi)=\{v_{i,1}(\psi),\ldots,v_{i,d_i}(\psi)\}$ with $v_{i,j}(\psi)$
$\beta$--H\"older continiuous in $\psi$ and $d_i={\rm dim}(V_i(\psi))$. We call
$\mathbf v(\psi)=\{v_{i,j}\}$
the basis on $\mathrm T_\psi\TTT^d$ given by the unioin of the $\mathbf
v_i(\psi)$, $i=1,\ldots, k$.

We can then ask whether we can
conjugate $dA_\e$ to a map $\mathcal B_\e(\psi,v):\mathrm
T\TTT^d\toitself$ of the form $\mathcal B_\e(\psi,v)=(A_0(\psi),\mathcal
L_\e(\psi)v)$
that
preserves the splitting $V_i$, $i=1,\ldots,k$. More precisely, we look for
invertible maps
$\mathcal B_\e=(A_0, \mathcal L_\e)$ and $\mathcal H_\e=(H_\e,\mathcal V_\e)$
such that
\begin{equation}\label{conjT}
\begin{aligned}
 \mathcal H_\e\circ \mathcal B_\e(\psi,v)&=(H_\e(A_0(\psi)),
\mathcal V_\e(A_0(\psi))\mathcal L_\e(\psi)v)\\ &= (A_\e(H_\e(\psi)), DA_\e(H_\e(\psi))v)=dA_\e\circ \mathcal H_\e(\psi,v)
\end{aligned}
\end{equation}
where $\mathcal L_\e:\mathrm T_\psi\TTT^d\mapsto \mathrm
T_{A_0(\psi)}\TTT^d$ and $\mathcal V_\e:\mathrm T_\psi\TTT^d\mapsto \mathrm
T_{H_\e(\psi)}\TTT^d$. Eq. \eqref{conjT} is the conjugancy condition.

Moreover, letting $P_i(\psi):\mathrm T_\psi\mathds T^d\to
V_i(\psi)$ be the projectors associated with the decomposition $(V_j)_{j=1}^k$
and setting
\begin{equation}\label{proj}
\mathcal L_{\e, i,j}(\psi)=P_i(A_0(\psi)) \mathcal
L_\e(\psi)\bigl|_{V_j(\psi)},\qquad
\mathcal
V_{\varepsilon,i,j}(\psi)=P_i(H_\e(\psi))\mathcal{V}_{\varepsilon}(\psi)\bigl|_
{ V_j(\psi) } ,
\end{equation}
we require that
\begin{equation}\label{pres}
\mathcal L_{\varepsilon,i,j}(\psi)=0\qquad\mathrm{and}\qquad\mathcal
V_{0,i,j}(\psi)=0\qquad\hbox{for  }i\ne j\;\mathrm{and}\;\psi\in \mathds
T^d.
\end{equation}
Eq. \eqref{pres} assures that $\mathcal H_\e$ preserves the splitting.
If such a function $\mathcal H_\e$ exists we call it a {\it partial
conjugation} between $dA_0$ and $dA_\e$. For simplicity sake, we will always
set $\mathcal L_0(\psi)=DA_0(\psi)$ in such a way that $\mathcal
V_0(\psi)=\mathrm{Id}$. This
approach is a generalization of the approach in Ref.~\onlinecite{GBG}, Chapter
10
where a similar construction is presented for the stable and unstable
directions of a linear automorphism of
$\TTT^2$.

To construct $\mathcal H_\e$ and $\mathcal B_\e$ in \eqref{conjT}, we will
assume that the maximum expansion rate generated by the action of $DA_0(\psi)$
on $V_i(\psi)$ is eventually strictly smaller than the minimum expansion rate on
$V_{i+1}(\psi)$, that is there exists $N_0>0$ such that for every $n\geq N_0$
and $\psi\in\TTT^d$ we have
\begin{equation}\label{conv}
\|\mathcal L_{0,i,i}^{n}(\psi)\|
\|(\mathcal L_{0,i+1,i+1}^{n}(\psi))^{-1}\|<1,\qquad
i=1,\ldots,k-1
\end{equation}
where $\mathcal L_{\e,i, i}^{n}(\psi)=\prod_{m=0}^{n-1}\mathcal L_{\e,i,
i}(A_0^m(\psi))$. When \eqref{conv} holds holds we say that
$\{V_1(\psi),\ldots,
V_k(\psi)\}$ form an hyperbolic splitting.

Thus to solve \eqref{conjT} we must find linear maps $\mathcal V_{\e,*,i}(\psi):
V_i(\psi)\mapsto \mathrm T_{H_{\e}(\psi)}\TTT^d$ and $\mathcal L_{\e,
i,i}(\psi):V_i(\psi)\mapsto V_i(A_0(\psi))$ that satisfy
\begin{equation}\label{vett-lin}
DA_{\e}(H_{\e}(\psi))\mathcal V_{\e,*,i}(\psi)=
\mathcal V_{\e,*,i}(A_0(\psi))\mathcal L_{\e,i,i}(\psi)
\end{equation}
where $\mathcal V_{\e,*,i}(A_0(\psi))=\mathcal V_{\e}(A_0(\psi))|_{V_i}$.
Observe that, if $\mathcal{V}_\e(\psi)$ and $\mathcal L_\e(\psi)$ are solutions
of \eqref{conjT} satisfying \eqref{pres}, then, for any invertible $\mathcal
G_\e(\psi):\mathrm
T_\psi\TTT^d\toitself$  such that $P_i(\psi)\mathcal
G_\e(\psi)|_{V_j}=0$, for $i\not= j$, also the new maps
\begin{equation}\label{ambi}
\tilde{\mathcal{V}}_\e(\psi)=\mathcal{V}_\e(\psi)\mathcal{G_\e(\psi)}
\quad\mathrm{and}\quad
\tilde{\mathcal{L}}_\e(\psi)=\mathcal
G_\e^{-1}(A_0\psi)\mathcal{L}_\e(\psi)\mathcal
G_\e(\psi)
\end{equation}
solve \eqref{conjT} and satisfy \eqref{pres}.

We look for solutions of \eqref{conjT} that are analytic in $\e$. To this
end we
use \eqref{ambi} and say say that $\mathcal V_\e$, $\mathcal L_{\e}$ are
analytic in $\e$ if there exists $\mathcal G_\e$ such that all entries of
$\mathcal V_\e(\psi)\mathcal G_\e(\psi)$ and of $\mathcal
G_\e^{-1}(A_0\psi)\mathcal{L}_\e(\psi)\mathcal
G_\e(\psi)$ are analytic once expressed as
matrices using the bases $\mathbf v(\psi)$, $\mathbf v(A_0(\psi))$ and $\mathbf
v(H_\e(\psi))$. In this way the analyticity of $\mathcal V_\e$, $\mathcal
L_{\e}$ does not depend on the choice of the bases $\mathbf v$. In practice, in
Section \ref{prova}, we will show that it
is enough to use \eqref{ambi} to fix $\mathcal
V_{\e,i,i}(\psi)=\mathrm{Id}$, once expressed in the bases $\mathbf v_i(\psi)$
and $\mathbf v_i(H_\e(\psi))$, to construct an analytic solution for
\eqref{vett-lin}.

We can now formulate our main theorem whose proof is in Section \ref{prova}.

\begin{theorem}\label{teo2}
Let $A_{\varepsilon}$ be defined as in \eqref{A} and $\{V_1,...,V_k\}$ an
hyperbolic
splitting of $\mathrm T_{\psi}\mathds{T}^d$ for $A_0$ satisfying \eqref{vettd}
and
\eqref{conv}.
Then, there exists $\bar \beta>0$ and $\bar\varepsilon(\beta)>0$ such that, for
$\beta<\bar \beta$ and $\varepsilon<\bar\varepsilon(\beta)$ there exists a
partial conjugation for $dA_\e$ that preserves the splitting $\{V_1,...,V_k\}$.
That is there exist
invertible linear maps $\mathcal V_\varepsilon(\psi):\mathrm T_\psi\mathds
T^d\mapsto\mathrm T_{H_{\varepsilon}(\psi)}\mathds T^d$ and $\mathcal
L_\varepsilon(\psi):\mathrm T_\psi\mathds T^d\mapsto\mathrm T_{A_0(\psi)}\mathds
T^d$, analytic in $\varepsilon$ and $\beta$-H\"{o}lder continuous in $\psi$ with
\[
\mathcal L_{\varepsilon,i,j}(\psi)=0\qquad\mathrm{and}\qquad\mathcal
V_{0,i,j}(\psi)=0\qquad\hbox{for  }i\ne j\;\mathrm{and}\;
\psi\in \mathds T^d,
\]
see \eqref{proj}, such that we get
\begin{equation}\label{secdeg}
DA_{\varepsilon}(H_{\varepsilon}(\psi))\mathcal{V}_{\varepsilon}(\psi)=
\mathcal{V}_{\varepsilon}(A_0(\psi))\mathcal{L}_{\varepsilon}(\psi)\,.
\end{equation}
\end{theorem}
\medskip

We can now set
\begin{equation}\label{push}
W_{\e,i}(\psi)=\mathcal{V}_{\e}(H^{-1}_{\e}(\psi))
V_i(H^{-1}_{\e}(\psi))
\end{equation}
so that $W_{\e,i}(\psi)\subset \mathrm T_\psi\TTT^d$, and obtain a solution of
\begin{equation*}\label{vett}
DA_{\varepsilon}(\psi)W_{\e,i}(\psi)=
W_{\e,i}(A_{\varepsilon}(\psi)),\qquad i=1,\ldots,k,
\end{equation*}
with $W_{0,i}(\psi)=V_i(\psi)$. Thus, the subspaces $W_{\e,i}(\psi)$ generate a
splitting of $\mathrm T_\psi\TTT^d$ that extend the splitting $V_i(\psi)$ for
$\e\not=0$. Observe though that $H_\e^{-1}(\psi)$, and thus $W_{\e,i}$, are not
in general analytic functions of $\e$ and cannot be constructed with the methods
used in Section \ref{prova}. On the other hand the subspaces
\begin{equation}\label{pb}
 V_{\e,i}(\psi)=W_{\e,i}(H_\e(\psi))=\mathcal{V}_{\e}(\psi)V_i(\psi)
\end{equation}
form a splitting analytic in $\e$. Thus we can recover analiticity for the
splitting by looking at the $DA_\e$-invariant  splitting $W_{\e,i}$ but
computing it on the moving point $H_\e(\psi)$.

A first application of Theorem \ref{teo2} is to the {\it minimal}
hyperbolic splitting
$W_0^{s,u}$. Since $A_\e$ is Anosov, we already know that $\mathrm T_\psi\TTT^2$
can be written as $W^u_\e(\psi)\oplus W^s_\e(\psi)$, with $W^{s,u}_\e$ invariant
under $dA_\e$. Theorem \ref{teo2} together with \eqref{pb} imply that
$V_\e^{s,u}(\psi)=W^{s,u}_\e(H_\e(\psi))$ form a splitting analytic in $\e$.
Setting
\[
 n^u(\psi)=\det\left(\mathcal V^*_{\e,*,u}(\psi)\mathcal
V_{\e,*,u}(\psi)\right)^{\frac12}
\]
and analogously for $n^s(\psi)$, we obtain an explicit expression for the
expansion and contraction rate on
$W_\e^{s,u}(\psi)$ as
\begin{equation}\label{expa}
 l^{s,u}_\e(\psi)=\det \left(\mathcal
L^{s,u}_\e(H^{-1}_\e(\psi))\right)\frac{n^{s,u}(H_\e^{-1}(\psi))}
{n^{s,u}(H_\e^{-1}(A_\e\psi))}
\end{equation}
where the determinant of $\mathcal L^{s,u}(\psi)$ is taken with respect to
the basis on
$W^{s,u}_\e(\psi)$ derived from \eqref{push}. It follows that
$l^{s,u}(H_\e(\psi))$ are analytic in $\e$ and H\"older continuous in
$\psi$.

Let $\mu_0$ be the normalized volume measure on $\mathds{T}^d$ and consider the
\emph{SRB} measure of $A_\varepsilon$ defined as
\begin{equation}\label{muinv}
\mu_{\varepsilon}=\lim_{n\to\infty}A_{\varepsilon}^n\mu_0.
\end{equation}
where the limit is intended in the weak sense. Since $A_{\varepsilon}$ is
Anosov, $\mu_{\varepsilon}$ exists and is ergodic. Moreover the fact that
$l^{u}_\e(H_\e(\psi))$ is analytic in $\e$ together with the explicit
construction of the SRB measure $\mu_\e$ as discussed in Ref.~\onlinecite{GBG}
Chapters 7
and 10, leads to the following result, whose proof is in Section \ref{SRB}.

\begin{theorem}\label{cor2}
Let $H_{\varepsilon}$ be the conjugation defined in \eqref{H} and
$\mu_{\varepsilon}$ the SRB measure defined in \eqref{muinv} for the Anosov
diffeomorphism $A_\e$ given in \eqref{A}, then there exists $\bar\e$ such that
for any   H\"{o}lder continuous function, $f: \mathds{T}^d\mapsto\mathds{R}$ we
have that
\begin{equation*}
\int_{\mathds{T}^{d}}f\circ
H^{-1}_{\varepsilon}(\psi)d\mu_{\varepsilon}(\psi)
\end{equation*}
is analytic in $\varepsilon$ for $\varepsilon<\bar\e$.
\end{theorem}

Of particular interest is the case when $A_0$ admits a \emph{maximal hyperbolic
splitting},
that is when there exists a basis $v_i(\psi)$ for $\mathrm
T_{\psi}\mathds T^d$ that satisfies
\begin{equation}\label{pb1}
DA_0(\psi)v_i(\psi)=\Lambda_i(\psi) v_i(A_0\psi)
\end{equation}
with $\Lambda_i:\TTT^2\mapsto \mathds R$ $\beta$--H\"older continuous for which
there exists $N_0$ such that, for every $n>N_0$ and $\psi\in\TTT^d$,
\begin{equation}\label{ord}
\|\Lambda_i^{n}(\psi)\|
\|(\Lambda_{i+1}^{n}(\psi))^{-1}\|< 1
\end{equation}
where $\Lambda_i^n(\psi)=\prod_{m=0}^{n-1}\Lambda_{i}(A_0^m(\psi))$ and
$i=1,\ldots, d-1$. In this
case $v_i(\psi)$  are also called the {\it Lyapunov vectors} and
$\lambda_i(\psi)$ the associated {\it Lyapunov numbers} for $A_0$.

From
Theorem \ref{teo2} we immediately get the following Corollary, formulated in
terms of the vectors $v_{\e,i}(\psi)=\mathcal V_{\e}(\psi)v_i(\psi)$.

\begin{corollary}\label{cor1}
Let $A_{\varepsilon}$ be defined as in \eqref{A} with $A_0$
admitting the maximal hyperbolic splitting $\{v_1(\psi), ..., v_d(\psi)\}$
of $\mathrm T_{\psi}\mathds{T}^d$ satisfying \eqref{pb1} and \eqref{ord}.
Then, there
exists $\bar \beta>0$ and $\bar\varepsilon(\beta)>0$ such
that, for $\beta<\bar \beta$ and
$\varepsilon<\bar\varepsilon(\beta)$ there exist $d$ linearly independent
vectors
$v_{\e,i}(\psi)\in \mathrm T_{H_\e(\psi)}\TTT^d$ and $d$ scalar
functions
$L_{\e,i}(\psi)$,
analytic in $\varepsilon$, $\beta$-H\"{o}lder continuous in $\psi$, such
that
\begin{equation}\label{secolar}
DA_{\varepsilon}(H_{\varepsilon}(\psi))v_{\e,i}(\psi)=
L_{\e,i}(\psi)v_{\e,i}(A_0\psi)
\end{equation}
where $ v_{i,0}(\psi)=v_i(\psi) , \  L_{i,0}(\psi)=
\Lambda_i(\psi)$.
\end{corollary}
\medskip

From \eqref{secolar} it follows that, we can set
\begin{equation}\label{Lambda}
w_{\e,i}(\psi)=
v_{\e,i}(H^{-1}_{\varepsilon}(\psi))\qquad\mathrm{and}\qquad
\Lambda_{\e,i}(\psi)=L_{\e,i}(H^{-1}_{\varepsilon}(\psi)),
\end{equation}
with $w_{\e,i}\in\mathrm T_\psi\TTT^d$, and get
\begin{equation}\label{eig}
DA_{\varepsilon}(\psi)w_{\e,i}(\psi)=
\Lambda_{\e,i}(\psi)
w_{\e,i}(A_{\varepsilon}(\psi)).
\end{equation}
Thus $w_{\e,i}(\psi)$ are the Lyapunov vectors and
$\Lambda_{\e,i}(\psi)$ the Lyapunov numbers for $A_\e$. As before,
they are in general only H\"older continuous in $\e$ while one recovers
analyticity by looking through the moving point $H_\e(\psi)$, that is by
considering $v_{\e,i}(\psi)=w_{\e,i}(H_\e(\psi))$ and
$L_{\e,i}(\psi)=\Lambda_{\e,i}(H_\e(\psi))$, see also \eqref{push} and
\eqref{pb}. Equations \eqref{Lambda} and \eqref{eig}, together with Theorem
\ref{cor2}, play a crucial role in the following discussion.


We can use \eqref{secolar} to study the Lyapunov exponent of
$A_\varepsilon$.
Observe that
\begin{equation*}
DA_{\varepsilon}^{n}=\prod_{k=0}^{n-1}DA_{\varepsilon}(A_{\varepsilon}^k)
\end{equation*}
so that the limit
\begin{equation*}\label{lyap}
\lambda_{\e,i}(\psi)=\lim_{n\to\infty}\frac{1}{n}
\log\left(\frac{\|DA_{\varepsilon}^{n}(\psi)
w_{\e,i}(\psi)\|}
{\|w_{\e,i}(\psi)\|}\right),
\end{equation*}
exists and it is $\mu_{\varepsilon}$-a.e. constant. This implies that
defining the $i$-th \emph{Lyapunov exponent} of $A_{\varepsilon}$ as
\begin{equation}\label{lyap1}
\lambda_{\e,i}=\int_{\mathds{T}^{d}}
\log(|\Lambda_{\e,i}(\psi)|)d{\mu}_{\varepsilon}(\psi)
\end{equation}
we have $\lambda_{\e,i}(\psi)=\lambda_{\e,i}$,
$\mu_{\e}$-a.e., where $\Lambda_{\e,i}$ is defined in \eqref{Lambda}.

Combining Theorem \ref{cor2} and Corollary \ref{cor1}  we get our final
result.

\begin{corollary}\label{arrivo}
Let $A_\e$ be defined as in \eqref{A} with $A_0$ admitting a maximal hyperbolic
splitting
$\{v_1(\psi),...,v_d(\psi)\}$
of $\mathrm T_{\psi}\mathds{T}^d$ satisfying \eqref{pb1} and \eqref{ord} and
let $L_{\e,i}$ be the solutions of
\eqref{secolar} provided by Corollary \ref{cor1}. Then, the $i$-th Lyapunov
exponent, defined in \eqref{lyap1}, is an analytic function of $\varepsilon$ for
$\varepsilon<\bar\varepsilon$ and can be written as
\begin{equation*}\label{formula}
\lambda_{\e,i}=\int_{\mathds{T}^{d}}
\log\left(|L_{\e,i}
(H^{-1}_{\varepsilon}(\psi))|\right)
d{\mu}_{\varepsilon}(\psi)
\end{equation*}
where $\mu_\e$ is the SRB measure of $A_\e$ defined in \eqref{muinv}.
\end{corollary}

\noindent{\bf Remark.} In the present paper we assumed the diffeomorphism
$A_\e$ to be analytic in $\psi$. In Ref.~\onlinecite{BKL} the authors show that,
in the
case in which $A_0$ is a linear automorphism of $\TTT^d$, for the conjugation
$H_\e$ to exists it is enough to require the perturbation $F$, see \eqref{A},
to be $C^1$ in $\psi$. Moreover their argument shows that, if $F$ is $C^k$ in
$\psi$ then $H_\e$ is $C^k$ in $\e$. Their proof follows a strategy based on
the Contraction Map Theorem and does not require power series expansions.
It should be possible to extend that strategy to a more general
$A_0$ like those considered here but not necessarily analytic. We also think
that a similar approach may be used to study the existence of $\mathcal L_\e$
ans $\mathcal V_\e$. In conclusion we believe that, with a different approach,
one may be able prove that if $A_0$ and $F$ in \eqref{A} are only $C^k$ in
$\psi$ then the Lyapunov exponent $\lambda_{i,\e}$ are $C^k$ in $\e$.
\medskip

The rest of the paper is organized as follows.
Section \ref{prova} contains the proof of Theorem
\ref{teo2} while in Section \ref{SRB} we prove Theorem \ref{cor2}. Finally in
Section
\ref{concl} we report some final considerations and possibilities for further
research.

\section{Proof of Theorem \ref{teo2}}\label{prova}

In this section we will lift both the argument and the image of $A_\e$ to
$\RRR^d$ as the universal cover of $\TTT^d$. Thus we see $F$ as a periodic and
analytic function from $\RRR^d$ into itself. With a slight
abuse of notation we use $A_\e$ to indicate both the functions from
$\TTT^d$ into itself and the function from $\RRR^d$ into itself. We believe no
confusion can arise. Similarly we will represent $\mathcal V_\e$ and
$\mathcal L_\e$
as functions from $\RRR^d$ to $\RRR^{d\times d}$ using the bases $\mathbf
v$. Observe that, in this representation, $V_i(\psi)$, and thus
$P_i(\psi)$, do not depend on $\psi$. We will thus call them $V_i$ and $P_i$
respectively.

To prove Theorem \ref{teo2}, we first show that
the solution $\mathcal V_{\varepsilon}$ and $\mathcal L_{\varepsilon}$ of
\eqref{secdeg} can be expanded in power series of $\e$ and the coefficients
of such an expansion can be
computed thanks to recursive relations. We then use these recursive relations
to show that the power series are convergent. This strategy and its
implementation are closely related but simpler than those used in
Ref.~\onlinecite{GBG} Chapter 10.

\subsection{Perturbative construction for the solutions of
(\ref{secdeg})}

Finally, without loss of generality, we will assume that \eqref{ord} holds with
$N_0=1$.
Indeed this can always be achieved by replacing $A_0$ with $A_0^{N_0}$.

We thus look for solutions of \eqref{secdeg} as power series in $\varepsilon$,
that is we write
\begin{equation}\label{VL}
\begin{aligned}
	\mathcal{V}_{\varepsilon}(\psi)&=\mathcal{V}_0+
	\sum_{n\geq1}\varepsilon^n\mathcal{V}^{(n)}(\psi), \\
	\mathcal L_{\varepsilon}(\psi)&=\mathcal L_0+
	\sum_{n\geq1}\varepsilon^n \mathcal L^{(n)}(\psi).
\end{aligned}
\end{equation}
with
\begin{equation}\label{ij}
\mathcal L_{\varepsilon,i,j}(\psi)=0\qquad\mathrm{and}\qquad\mathcal
V_{0,i,j}(\psi)=0\qquad\hbox{for  }i\ne j\;\mathrm{and}\;\psi\in \mathds
T^d.
\end{equation}
Moreover we can take $\mathcal{V}_0(\psi)=\mathrm{Id}$ and
$\mathcal{L}_0(\psi)=DA_0(\psi)$.
As discussed before Theorem \ref{teo2}, we will use \eqref{ambi} to impose
the following {\it normalization} condition
\begin{equation*}\label{norm}
P_i\mathcal{V}_{\varepsilon}(\psi)|_{V_i}=
\mathrm{Id},
\end{equation*}
so that
\begin{equation}\label{orth}
	\mathcal{V}_{i,i}^{(n)}(\psi)=0
\end{equation}
for all $n>0$, where
$\mathcal{V}^{(n)}_{i,j}(\psi)=P_i\mathcal{V}^{(n)}(\psi)\rvert_{V_j}$.

From Ref.~\onlinecite{GBG} and Ref.~\onlinecite{Falco} it follows that can be
written as
\begin{equation*}\label{analitica}
H_{\varepsilon}(\psi)=\psi + \sum_{n\geq1}\varepsilon^n h^{(n)}(\psi)
\end{equation*}
with $\|h^{(n)}\|_{\beta}\leq C_\beta^n$ for every $n>1$ and any $\beta$ small
enough, with $C_\beta>0$ depending on $\beta$. Thus there exist
$\Phi^{(n)}(\psi)$,
$n\geq 0$, such that
\begin{equation}\label{Phi}
	DF(H_{\e}(\psi))=\sum_{n\geq0}\e^n\Phi^{(n)}(\psi)
\end{equation} with $\Phi^{(0)}(\psi)=DF(\psi)$.

We can now plug \eqref{VL} in \eqref{secdeg}, expand both sides in power of
$\varepsilon$ and equate the resulting terms of the same order. The zero order
equation is trivially satisfied while the first order of \eqref{secdeg} reads
\begin{equation}\label{f1}
	DA_0(\psi)\mathcal{V}^{(1)}(\psi)+DF(\psi)=
	\mathcal{V}^{(1)}(A_0(\psi))DA_0(\psi)+\mathcal{L}^{(1)}(\psi),
\end{equation}
Note that in \eqref{f1} both $\mathcal{L}^{(1)}(\psi)$ and
$\mathcal{V}^{(1)}(\psi)$ are unknown.

Projecting \eqref{f1} with $P_i$ while restricting on $V_i$ we
obtain
\begin{equation}\label{L1}
\mathcal{L}^{(1)}_{i,i}(\psi)=P_iDF(\psi)|_{V_i}
\end{equation}
where $\mathcal{L}^{(1)}_{i,j}(\psi)=P_i\mathcal{L}^{(1)}(\psi)|_{V_j}$ and we
have used \eqref{orth} and that $DA_0(\psi)P_i=P_iDA_0(\psi)$.
Calling $D_jF_i(\psi)=P_iDF(\psi)|_{V_j}$,
we get
\[
\mathcal{L}_{i,i}^{(1)}(\psi)=D_iF_i(\psi).
\]
On the other hand, still projecting with $P_i$ but
restricting on $V_{j}$ for $j\not=i$,  we obtain
\begin{equation}\label{sha}
DA_{0,i}(\psi)\mathcal{V}^{(1)}_{i,j}(\psi)+D_{i}F_j(\psi)=
\mathcal{V}_{i,j}^{(1)}(A_0(\psi))DA_{0,j}(\psi).
\end{equation}
where we used that $\mathcal{L}^{(1)}_{i,j}(\psi)=0$.

To simplify the notation, we set $\psi_m=A_0^m\psi$,
$\mathds{Z^+}=\mathds{N}\cup\{0\}$, $\mathds{Z^-}=-\mathds{N}$ and
$\omega_{ij}=-$ if $i<j$ and $\omega_{ij}=+$ otherwise.
The solution of \eqref{sha} can now be written as
\begin{equation}\label{primo}
	\mathcal{V}_{i,j}^{(1)}(\psi)=-\sum_{m\in\mathds{Z}^{\omega_{ij}}}
	(\mathcal{L}_{0,i,i}^{m+1}(\psi_m))^{-1}
	D_iF_j(\psi_m)\mathcal{L}_{0,j,j}^{m}(\psi_m)
\end{equation}
where, being $\mathcal{L}_0=DA_0$, the series in \eqref{primo} is convergent due
to our definition of $\mathds{Z}^{\omega_{ij}}$ and the ordering of the
$\mathcal{L}_{0,i,i}$ in \eqref{conv}.

Thus with \eqref{L1} and \eqref{primo} we have an explicit expressions for the
first order term in \eqref{VL}. For the higher order terms we find, using the
expansion \eqref{Phi} for $DF(H_{\e}(\psi))$, the following equation
\begin{equation}\label{high}
\mathcal{L}_0(\psi)\mathcal{V}^{(n)}(\psi)+
\sum_{p=0}^{n-1}\Phi^{(n-p-1)}(\psi)\mathcal{V}^{(p)}(\psi)=
\mathcal{L}^{(n)}(\psi) + \sum_{p=1}^{n}\mathcal{V}^{(p)}(A_0
(\psi))\mathcal{L}^{(n-p)}(\psi).
\end{equation}
Setting
$\Phi_{i,l}^{(n)}(\psi)=
P_i\Phi^{(n)}(\psi)|_{V_l}$ and
\begin{equation}\label{eta}
	\eta_{ij}^{(n)}(\psi)=\sum_{l=1}^k
	\sum_{p=0}^{n-1}\Phi_{i,l}^{(n-p-1)}(\psi)\mathcal{V}_{l,j}^{(p)}(\psi),
\end{equation}
projecting \eqref{high} with $P_i$ while restricting to
$V_i$, and using \eqref{orth} we get
\begin{equation}\label{L}
\mathcal{L}_{i,i}^{(n)}(\psi)=\sum_{j=1}^k
\sum_{p=0}^{n}\Phi_{i,l}^{(n-p-1)}(\psi)\mathcal{V}_{l,i}^{(p)}(\psi)=\eta_{ii}^
{(n)}(\psi).
\end{equation}
Still projecting with $P_i$ but restricting on $V_{j}$ for $j\neq
i$, and using \eqref{ij} we get
\begin{equation*}\label{high2}
	\mathcal{L}_{0,i,i}(\psi)\mathcal{V}^{(n)}_{i,j}(\psi)+
\eta_{ij}^{(n)}(\psi)=
	\mathcal{V}_{i,j}^{(n)}(A_0(\psi))\mathcal{L}_{0,j,j}(\psi)+
\sum_{p=1}^{n-1}\mathcal{V}^{(p)}_{i,j}(A_0
	(\psi))\mathcal{L}^{(n-p)}_{j,j}(\psi),
\end{equation*}
whose solution is
\begin{equation}\label{V}
	\mathcal{V}^{(n)}_{i,j}(\psi)=\sum_{m\in\mathds{Z}^{\omega_{ij}}}
(\mathcal{L}_{0,i,i}^{m+1}(\psi_m))^{-1}
	\left[\sum_{p=1}^{n-1}\mathcal{V}^{(p)}_{i,j}(A_0
(\psi))\mathcal{L}^{(n-p)}_{j,j}(\psi)-\eta_{ij}^{(n)}(\psi)\right]
\mathcal{L}_{0,j,j}^{m}(\psi_m).
\end{equation}
As for the first order, the convergence of the series is provided by our
definition of $\om_{ij}$ and the ordering of $\mathcal{L}_{0,i,i}$ in
\eqref{conv}.

Observe that the r.h.s. of \eqref{L} and \eqref{V} depends only on
$\mathcal{V}^{(p)}_i$ with $p<n$. In this way \eqref{L} and \eqref{V} provide a
recursive construction of the coefficient of the power series expansion
\eqref{VL}. We will
show in the following section that these series converge and thus provide a
solutions for
\eqref{secdeg}.

\subsection{Proof of convergence of the perturbative series}

Theorem \ref{teo2} will follow directly from the following technical Lemmas.

\begin{lemma}\label{boundfi}
For $\beta$ small enough there exists a constant $B_\beta>0$ such that
\[
\|\Phi^{(n)}\|_{\beta}\leq B_\beta^{n+1},
\]
for all $n\geq 0$.
\end{lemma}

\begin{proof}
Observe that from Ref.~\onlinecite{Falco} it follows that there exists $r>0$
such that,
for every $\psi$, $\mathrm{f}(\varepsilon,\psi)=F(H_{\varepsilon}(\psi))$ is
analytic in $\e$ in a disk of radius $2r>0$ in the complex plane.
Using Cauchy formula on the path $\gamma(t)=r e^{2\pi i t}$, $0\leq t\leq 1$ and calling $R=r^{-1}$,
we get
\[
\begin{aligned}
\|\Phi^{(n)}\|_{\beta}=&\left\|\frac{1}{2\pi i
}\oint_{\gamma}\frac{\mathrm{f}(\e,\cdot)}{\e^{n+1}}d\e\right\|_{\beta}\leq
\frac{1}{2\pi}\oint_{\gamma}\frac{\|\mathrm{f}(\varepsilon,\cdot)\|_
	{
		\beta}}{|\e|^{n+1}} d\e\\
=& R^{n+1}\|\mathrm{f}(\varepsilon,\cdot)\|_{\beta}=R^{n+1}\|F\circ
H_{\e}\|_{\beta}.
\end{aligned}
\]
Then, the H\"{o}lder continuity of $F\circ H_{\e}$, uniform
with respect to $\e$ in the disk of radius $2r$, gives the thesis.
\end{proof}

In order to bound
$\|\mathcal{V}_{i,j}^{(n)}\|_\beta$ we observe that, calling
$\Omega=\max\{\|DA_0\|_\infty,\|DA_0^{-1}\|_\infty^{-1}\}$,
for any $f$ and any
$m\in\mathds{Z}$ we get
\begin{equation}\label{holderA}
	\|f\circ A_0^m\|_\beta\leq \Omega^{\beta|m|}\|f\|_\beta.
\end{equation}
Moreover observe that for any $f$ and $g$ we have
\begin{equation}\label{algebra}
|fg|_\beta\leq\|f\|_\infty|g|_\beta+|f|_\beta\|g\|_\infty,\qquad
\|fg\|_\beta\leq\|f\|_\beta\;\|g\|_\beta\, .
\end{equation}

This leads to the following Lemma.

\begin{lemma}\label{sequence}
	For $\beta$ small enough
	there exists a constant $E_\beta>0$ such that
	\begin{equation}\label{induction2}
		\|\mathcal{V}^{(n)}_{ij} \|_{\beta}\leq \tilde{E}_\beta^n,\qquad\quad
		\|\mathcal{L}^{(n)}_{ij} \|_{\beta}\leq \tilde{E}_\beta^n
	\end{equation}
	for all $n\geq 0$.

\end{lemma}

\begin{proof}
From \eqref{holderA} and first of \eqref{algebra} it follows that
\[
\|\mathcal{L}_{0,i,i}^{m}\|_\beta\leq\frac{\Omega^{\beta|m|}}
{|\Omega^\beta-1|}\|\mathcal{L}_{0,i,i}\|_\beta
\|\mathcal{L}_{0,i,i}\|_\infty^{m-1}\qquad\quad
\|(\mathcal{L}_{0,i,i}^{m})^{-1}\|_\beta\leq\frac{\Omega^{\beta|m|}}
{|\Omega^\beta-1|}\|\mathcal{L}_{0,i,i}\|_\beta
\|(\mathcal{L}_{0,i,i})^{-1}\|_\infty^{m+1}
\]
Using \eqref{V} and \eqref{holderA} we get
\begin{equation}\label{stima1}
	\begin{aligned}
	\|\mathcal{V}_{i,j}^{(n)}\|_\beta\leq&
	K_\beta\sum_{m\in\mathds{Z}^{\omega_{ij}}}\sum_{p=1}^{n-1}
	\left(\big\|\mathcal{L}_{0,i,i}^{(n-p)}\circ A_0^m\;
	\mathcal{V}_{i,j}^{(p)}\circ A_0^{m+1}\big\|_\beta+
	\|\eta_{i,j}^{(n)}\circ A_0^m\|_\beta\right)\\
	&\phantom{K_\beta\sum_{m\in\mathds{Z}^{\omega_{ij}}}\sum_{p=1}^{n-1}}
	\left(\Omega^{2\beta}\|\mathcal{L}_{0,j,j}\|_\infty
	\|(\mathcal{L}_{0,i,i})^{-1}\|_\infty\right)^m\\ \leq &
	K_\beta\sum_{m\in\mathds{Z}^{\omega_{ij}}}
	\left(\|\mathcal{L}_{0,j,j}\|_\infty
	\|(\mathcal{L}_{0,i,i})^{-1}\|_\infty\right)^m
	\Omega^{3\beta(|m|+1)}\sum_{p=1}^{n-1}
	\left(\big\|\mathcal{V}_{i,l}^{(p)}\mathcal{L}_{lj,}^{(n-p)}
	\big\|_\beta+\|\eta_{i,j}^{(n)}\|_\beta\right)
	\end{aligned}
\end{equation}
where
\[
K_\beta=\|\mathcal{L}_{0,i,i}\|_\beta\|\mathcal{L}_{0,j,j}\|_\beta
\|\mathcal{L} _ {0,j}\|_\infty^2/|\Omega^\beta-1|^2
\]
while \eqref{eta} and \eqref{L}, together with Lemma \ref{boundfi}, give
\begin{equation}\label{stima2}
	\|\mathcal{L}_{i,i}^{(n)}\|_\beta\,,\;\|\eta_{i,j}^{(n)}\|_\beta\leq
	\sum_{p=0}^{n-1}\sum_{l=1}^k\|\mathcal{V}_{i,l}^{(p)}\|_\beta
	B_\beta^{n-p}.
\end{equation}
Moreover observe that for $\beta$ small enough we have
\begin{equation}\label{stima3}
\begin{aligned}
	K_\beta\Omega\sum_{m\in\mathds{Z}^{\omega^{ij}}}
	\left(\|\mathcal{L}_{0,i,i}\|_\infty\|(\mathcal{L}_{0,j,j})^{-1}
	\|_\infty\right)^m
	\Omega^{3\beta|m|}=&\\
	\frac{K_\beta\Omega\|\mathcal{L}_{0,i,i}
	\|_\infty\|(\mathcal{L}_{0,j,j})^{-1}\|_\infty}
	{\omega_{ij}(1-\|\mathcal{L}_{0,i,i}\|_\infty
	\|(\mathcal{L}_{0,j,j})^{-1}\|_\infty
	\Omega^{3\omega_{ij}\beta})}=&K_{1,\beta}.
\end{aligned}
\end{equation}
Combining \eqref{stima1}, \eqref{stima2} and \eqref{stima3} we get
\begin{equation}\label{stimavettori}
	\begin{aligned}
	\|\mathcal{V}_{i,j}^{(n)}\|_\beta\leq
	K_{1,\beta}\left(\sum_{p=1}^{n-1}
	\sum_{q=0}^{p-1}\sum_{l=1}^k\|\mathcal{V}_{il}^{(q)}\|_\beta
	\|\mathcal{V}_{i,j}^{(n-p)}\|_\beta B_\beta^{p-q}
	+\sum_{p=0}^{n-1}\|\mathcal{V}_{il}^{(p)}\|_\beta
	B_\beta^{n-p}\right)
	\end{aligned}
\end{equation}

We can now prove that
\begin{equation}\label{induction}
	\|\mathcal{V}_{i,j}^{(n)}\|_\beta\leq  2^{n-1}C_{n-1}E_\beta^n.
\end{equation}
with $E_\beta=K_{1,\beta}B_\beta k$ where $C_n$ are the Catalan's number defined
recursively as 	\begin{equation*}\label{Catalan}
\begin{cases}
	C_0=1,  \\
	\displaystyle{C_n=\sum_{j=0}^{n-1}C_{n-1-j}C_j}.
	\end{cases}
\end{equation*}
It is easy to see that \eqref{induction} holds for $n=1$. Assuming
\eqref{induction} to be true for all integers $\leq n-1$ we plug it in
\eqref{stimavettori} and we obtain
\begin{equation*}
	\begin{aligned}
	\|\mathcal{V}_{i,j}^{(n)}\|_\infty\leq&
	kK_{1,\beta}B_\beta E_\beta^{n-1}
	\Big(\sum_{p=1}^{n-1}2^{n-p-1}C_{n-p-1}\sum_{q=0}^{p-1}C_{q-1}2^{q-1}+
	\sum_{p=0}^{n-1}C_{p-1}\Big)\\ 
	\leq&E_\beta^{n}\Big(\sum_{p=1}^{n-1}2^{n-2}C_{n-p-1}C_{p-1}+
	\sum_{p=0}^{n-1}C_{p-1}\Big) \\ &\leq
	E_\beta^{n}\Big(2\sum_{p=1}^{n-1}2^{n-2}C_{n-p-1}C_{p-1}\Big)=
	E_\beta^{n}2^{n-1}C_{n-1},
	\end{aligned}
\end{equation*}
where we have used the fact that
\begin{equation}\label{sumal}
	\sum_{q=0}^{p-1}2^{q-1}C_{q-1}\leq2^{p-1} C_{p-1}
\end{equation}
and that for all $n\geq 4$
\begin{equation}
	\sum_{p=0}^{n-1}
	2^{p-1}C_{p-1}\leq \sum_{p=1}^{n-1}2^{n-2}C_{n-p-1}C_{p-1}.
\end{equation}
This complete the inductive step and since $C_n\sim \frac{4^n}{\sqrt{n^3\pi}}$
we get $\tilde{E}_{\beta}=8E_{\beta}$.
Finally, the second bound in \eqref{induction2} follows from \eqref{stima2}
and \eqref{sumal}.

\end{proof}

From Lemma \ref{sequence}  we get that the series \eqref{VL} converge for
$|\varepsilon| < \tilde{E}_\beta^{-1}$ and thus define $\beta$-H\"{o}lder
continuous solutions for \eqref{secdeg}.

Finally choosing $\bar\varepsilon(\beta)=(8k\tilde{E}_\beta)^{-1}$ we see
that, for
$|\varepsilon|\leq \bar\varepsilon(\beta)$, Lemma \ref{sequence}, and
\eqref{VL} imply that $\|\mathcal{V}_{i,j,\varepsilon}\|_\infty\leq k^{-1}$ while, by definition, $\mathcal{V}_{i,\e,i}(\psi)=\mathrm{Id}$.
Therefore, by Gershgorin circle theorem, the matrix
$\mathcal{V}_{\varepsilon}(\psi)$ is invertible for every $\psi$. This
conclude the proof of Theorem
\ref{teo2}.

\section{Proof of Theorem \ref{cor2}}\label{SRB}

To proof Theorem \ref{cor2} we closely follow Chapter 10 of
Ref.~\onlinecite{GBG}. We
report here only the main steps of the argument and refer the reader to the
relevant statements in Ref.~\onlinecite{GBG} for the technical parts.

The main tool to study $\mu_\e$ is the construction of a suitable {\it symbolic
dynamics} for $A_0$. Let $\mathbf Q=\{Q_1,\ldots, Q_J\}$ be a partition of
$\TTT^d$
in $J$ measurable set $Q_i$, $i=1\ldots,J$. Since $A_0$ is hyperbolic, if
$\mathrm{diam}(Q_i)$ is small enough we have that, for every
$\ul\sigma\in\{1,\ldots,J\}^{\ZZZ}$, with
$\ul\sigma=(\ldots,\sigma_{-1},\sigma_0,\sigma_1,\ldots)$, the set
\[
 U_{\ul\sigma}=\bigcap_{i=-\infty}^{\infty}A_0^i(Q_{\sigma_i})
\]
contains at most one point. We call this point $X_0(\ul\sigma)$ so that $X_0$,
the {\it symbolic map}, is a function from a subset $M\subset
\{1,\ldots,J\}^{\ZZZ}$ to $\TTT^d$. Setting $v(\ul\sigma,\ul\tau)=\min\{|i|\,;\,
\sigma_i\not=\tau_i)$ and $\delta(\ul\sigma,\ul\tau)=2^{-v(\ul\sigma,\ul\tau)}$
we
see that $X_0$ is an H\"older continuous function with respect to the distance
$\delta$.

In general, a symbolic dynamics is not very useful since the set $M$ can be very
complex. Sometime it is possible to construct partitions for which the set $M$
is relatively simple. Let $T\in\{0,1\}^{J\times J}$ be defined as
$T_{\sigma,\tau}=1$  if $A_0(Q_\sigma)\cap Q_\tau\not=\emptyset$ and 0
otherwise. We say that a partition $\mathbf Q$ is a Markovian pavement with
compatibility matrix $T$ when $\ul\sigma\in M$ if and only if
$T_{\sigma_i,\sigma_{i+1}}=1$ for every $i\in\ZZZ$.
%
Since
$A_0$ is Anosov, it follows from Ref.~\onlinecite{Bo70} (see also
Ref.~\onlinecite{Ru78}) that it
admits a Markovian
pavement $\mathbf Q$ with a mixing transition matrix $T$, that is there
exists $r$ such that $\left(T^r\right)_{i,j}>0$ for every
$i,j\in\{1,\ldots,J\}$. In this case we call $M=\Sigma_T$ (see
Ref.~\onlinecite{GBG}
Chapters 4.1 and 4.2).

Observe now that if $\mathbf Q$ is a Markovian pavement for $A_0$ then
$H_\e(\mathbf Q)=\{H_\e(Q_1),\ldots,H_\e(Q_J)\}$ is a Markovian pavement for
$A_\e$ with symbolic map $X_\e(\ul\sigma)=H_\e(X_0(\ul\sigma))$. Finally we have
that $A_\e(X_\e(\ul\sigma))=X_\e(s(\ul\sigma))$ where
$s:\Sigma_T\mapsto\Sigma_T$ is the {\it shift map} defined by
$(s(\ul\sigma))_i=\sigma_{i+1}$. Thus $X_\e$ can be seen as a conjugation
between $A_\e$ acting on $\TTT^d$ and the {\it subshift of finite type}
generated by $s$ acting on $\Sigma_M$.

We can now lift the one parameter family of measure $\mu_\e$ on $\TTT^d$ to
the one parameter family $m_\e=X_\e^*\mu_\e$ on $\Sigma_T$. Moreover, it
follows from \eqref{expa} that
$A^u_\e(\sigma)=\log(|l^u_\e(X_{\e}(\sigma))
|)$ is analytic in $\e$ and H\"older continuous in $\sigma$.
Given any $\ul\tau\in\Sigma_T$ (the {\it boundary conditions}), we define
\[
\Psi_{\e,[-n,n]}(\sigma_{-n},\ldots,\sigma_n)=
A^u_\e(\ul\sigma_{[-n+1,n+1]}\wedge \ul\tau)-A^u_\e(\ul\sigma_{[-n,n]}\wedge
\ul\tau)
\]
where $(\ul\sigma_{[-n,n]}\wedge \ul\tau)_i=\sigma_i$ if $i\in [-n,n]$ while
$(\ul\sigma_{[-n,n]}\wedge \ul\tau)_i=\tau_i$ otherwise\footnote{Care must be
paid to the fact that the mixing time $r$ of $T$ may be strictly positive. In
such a case the definition of $\ul\sigma_{[-N,N]}\wedge \ul\tau$ is slightly
more complex. See (4.3.10) in Ref.~\onlinecite{GBG} and surrounding
discussion for
more details.}. Moreover we set $\Psi_{\e,X}(\ul\sigma_X)=0$ if $X\not=[-n,n]$
for some $n$. It follows that $\Psi_\e$ is a Fisher Potential according to
Definition 7.3.1 in Ref.~\onlinecite{GBG} and $\|\Psi_\e\|_\kappa\leq C\e$ for
suitable
constants $C$ and $\kappa$ (see (7.3.2) in Ref.~\onlinecite{GBG}).
Thus, using Proposition 4.3.2 and 6.3.3 of Ref.~\onlinecite{GBG} we get that
$m_\e$ is the
Gibbs state on $\Sigma_M$ generated by the potential $\Psi_\e$, see Definition
5,1,2 in Ref.~\onlinecite{GBG}.

Finally we apply Proposition 7.3.2 of Ref.~\onlinecite{GBG} and obtain that
$m_\e$ is
analytic in $\e$, that is for every H\"older continuous function $\hat
f:\Sigma_T\mapsto \RRR$ we have that
\[
 \int_{\Sigma_T}\hat f(\ul\sigma)dm_\e(\ul\sigma)
\]
is an analytic function of $\e$ for $\e$ small enough,  (see also
Ref.~\onlinecite{Co80}).
To
conclude the proof we observe that if $f:\TTT^d\mapsto\RRR$ is H\"older
continuous then
\[
\int_{\mathds{T}^{d}}f(
H^{-1}_{\e}(\psi))d\mu_{\e}(\psi)= \int_{\Sigma_T}
f(X_0(\ul\sigma))dm_\e(\ul\sigma)
\]
with $f(X_0(\ul\sigma))$ H\"older continuous in $\ul\sigma$.

\section{Conclusion}\label{concl}

In this section we present further details about a possible application of the
present work, by discussing a concrete and very simple example obtained by
coupling two linear automorphisms $B_1$ and $B_2$ of $\mathds{T}^2$. More
precisely consider the system $A_\varepsilon$ acting on $\mathds T^4=\mathds
T^2\times\mathds T^2$ and given by \eqref{A} with $A_0(\psi)=A_0\psi\mod 2\pi$
and
\begin{equation*}
A_0=	\begin{pmatrix}
	B_1 & 0 \\
	0 & B_2
\end{pmatrix}\,.
\end{equation*}
Assume moreover that, for $(\psi_1,\psi_2)\in \mathds T^2\times\mathds T^2$, we
have
\begin{equation}\label{split}
	F(\psi_1,\psi_2)=\begin{pmatrix} G_1(\psi_1) \\ G_2(\psi_2)\end{pmatrix}\, .
\end{equation}
with $G_i:\TTT^2\mapsto \RRR^2$.
If $\lambda_{i,1,\varepsilon}$ and $\lambda_{i,2,\varepsilon}$, $i=1,2$, are the
Lyapunov exponents of the dynamical system $B_i+\varepsilon G_i \mod 2\pi$ on
$\mathds \TTT^2$ then, from Ref.~\onlinecite{Young}, we know that the Hausdorff
dimension
of the SRB measure $\mu_\varepsilon$ is
\begin{equation}\label{You}
	{\rm dim}_{\rm HD}(\mu_\varepsilon)=
	4-\frac{\lambda_{1,1,\varepsilon}+\lambda_{1,2,\varepsilon}}
	{\lambda_{1,1,\varepsilon}}-
	\frac{\lambda_{2,1,\varepsilon}+\lambda_{2,2,\varepsilon}}
	{\lambda_{2,1,\varepsilon}}
\end{equation}
while the KY conjecture gives
\begin{equation}\label{KYe}
	{\rm dim}_{\rm L}(\mu_\varepsilon)=
	4-\frac{\lambda_{1,1,\varepsilon}+\lambda_{1,2,\varepsilon}+
	\lambda_{2,1,\varepsilon}+\lambda_{2,2,\varepsilon}}
	{\lambda_{1,1,\varepsilon}}
\end{equation}
From the present works it follows that if $B_1\not=B_2$, so that
$\lambda_{1,i,0}\not=\lambda_{2,i,0}$, then, for $\epsilon$ small, \eqref{You}
and \eqref{KYe} give different results and the KY conjecture cannot be true. As
stated in the introduction, it would suffice to show that ${\rm dim}_{\rm
HD}(\mu_\varepsilon)$ is continuous in $\varepsilon$ to obtain that the KY
conjecture is not verified in an open neighbor of $A_0$. We think we can prove
that ${\rm dim}_{\rm HD}(\mu_\varepsilon)$ is actually an analytic function of
$\varepsilon$.

It is interesting to notice that if $B_1=B_2$ and $G_1=G_2$ then \eqref{You}
and \eqref{KYe} agree. This suggests that, at least in this context, the KY
conjecture can depend on some extra symmetry of the system. A natural conjecture
is that the KY conjecture holds true when
$B_1=B_2$ and
\[
F(\psi_1,\psi_2)= F(\psi_2,\psi_1)\, .
\]
Notwithstanding this, if $F$ is not of the form \eqref{split} while $B_1=B_2$,
our results only tell us that
$\lambda_{1,1,\varepsilon}+\lambda_{2,1,\varepsilon}$ is a smooth function of
$\varepsilon$. Thus a first step toward a proof of our conjecture is to extend
the analysis in this paper to understand when and how a perturbation can resolve
the degeneracy of the Lyapunov exponents.


\providecommand{\bysame}{\leavevmode\hbox to3em{\hrulefill}\thinspace}
\providecommand{\MR}{\relax\ifhmode\unskip\space\fi MR }
\providecommand{\MRhref}[2]{%
  \href{http://www.ams.org/mathscinet-getitem?mr=#1}{#2}
}
\providecommand{\href}[2]{#2}

\end{document}